\documentclass[11pt]{amsart}

\title[Arc complexes, sphere complexes and Goeritz groups]
{Arc complexes, sphere complexes and Goeritz groups}

\author{Sangbum Cho}
\thanks{The first-named author is supported by Basic Science Research Program through the National Research Foundation of Korea (NRF) funded by the Ministry of Education, Science and Technology (2012006520).}
\address{Department of Mathematics Education, Hanyang University, Seoul 133-791,
Korea}
\email{scho@hanyang.ac.kr}

\author{Yuya Koda}
\thanks{The second-named author is supported by JSPS Postdoctoral Fellowships for Research Abroad, and
by the Grant-in-Aid for Young Scientists (B), JSPS KAKENHI Grant Number 26800028.}

\address{
Department of Mathematics \newline 
\indent Hiroshima University, 1-3-1 Kagamiyama, Higashi-Hiroshima, 739-8526, Japan}
\email{ykoda@hiroshima-u.ac.jp}

\author{Arim Seo}
\thanks{}
\address{Department of Mathematics Education, Korea University, Seoul 136-701, Korea}
\email{arimseo@korea.ac.kr}

\usepackage{amsfonts,amsmath,amssymb,amscd}
\usepackage{amsthm}
\usepackage{latexsym}
\usepackage{graphicx}
\usepackage{psfrag}
\usepackage{color}
\usepackage{pinlabel}
\usepackage{xypic}
\usepackage[all]{xy}

\theoremstyle{plain}
\newtheorem*{theorem*}{Theorem}
\newtheorem*{lemma*} {Lemma}
\newtheorem*{corollary*} {Corollary}
\newtheorem*{proposition*}{Proposition}
\newtheorem*{conjecture*}{Conjecture}
\newtheorem{theorem}{Theorem}[section]
\newtheorem{lemma}[theorem]{Lemma}
\newtheorem{corollary}[theorem]{Corollary}
\newtheorem{proposition}[theorem]{Proposition}

\newtheorem{claim}{Claim}

\theoremstyle{remark}

\theoremstyle{definition}

\newcommand{\Natural}{\mathbb{N}}
\newcommand{\Integer}{\mathbb{Z}}

\newcommand{\MCG}{\mathrm{MCG}}
\newcommand{\Homeo}{\mathrm{Homeo}}

\newcommand{\Nbd}{\operatorname{Nbd}}
\newcommand{\st}{\mathit{st}}



\begin{document}

\maketitle

\begin{abstract}
We show that if a Heegaard splitting is obtained by gluing a splitting of Hempel distance at least $4$ and the genus-$1$ splitting of $S^2 \times S^1$, then the Goeritz group of the splitting is finitely generated. 
To show this, 
we first provide a sufficient condition for a full subcomplex of the arc complex for a compact orientable surface
to be contractible, which generalizes the result by Hatcher that the arc complexes are contractible. 
We then construct infinitely many Heegaard splittings, including the above-mentioned Heegaard splitting, 
for which suitably defined complexes of Haken spheres are contractible. 
\end{abstract}

%
%

\section*{Introduction}

Let $\Sigma_{g, n}$ be a compact connected orientable surface of genus $g$ with $n$ holes,
where $n \geqslant 3$ if $g = 0$, and $n \geqslant 1$ if $g \geqslant 1$.
As an analogue of the curve complex, the {\it arc complex} $\mathcal{A}_{g, n}$ of $\Sigma_{g, n}$
is defined to be the simplicial complex whose vertices are isotopy classes of
essential arcs in $\Sigma_{g, n}$ and whose $k$-simplices are
collections of $k+1$ vertices represented by pairwise disjoint and non-isotopic arcs in $\Sigma_{g, n}$.
In \cite{Hat91}, Hatcher proved that the complex $\mathcal{A}_{g, n}$ is contractible.
See also Irmak-McCarthy \cite{IM10} and Korkmaz-Papadopoulos \cite{KP10} for related works on arc complexes.

In Section \ref{sec:Arc complexes} in this paper, we provide a useful sufficient condition for a full subcomplex of the arc complex
to be contractible (Theorem \ref{thm:contractibility of arc complexes}).
Since the arc complex $\mathcal{A}_{g,n}$ itself satisfies this condition, it is contractible, which gives an updated proof for the Hatcher's result.
Moreover, we also show that
the full subcomplex $\mathcal{A}_{g, n}^*$
of $\mathcal{A}_{g, n}$, with $n \geqslant 2$, spanned by vertices of arcs connecting different boundary components is contractible.

A genus-$g$ {\it Heegaard splitting} of a closed orientable $3$-manifold $M$ is a decomposition of the manifold into two handlebodies of the same genus $g$.
That is, $M = V \cup W$ and $V \cap W = \partial V = \partial W = \Sigma$, where $V$ and $W$ are handlebodies of genus $g$ and $\Sigma$ is their common boundary surface.
We simply denote by $(V, W; \Sigma)$ the splitting, and call the surface $\Sigma$ the {\it Heegaard surface} of the splitting.
It is well known that every closed orientable $3$-manifold admits a genus-$g$ Heegaard splitting for some genus $g \geq 0$.
Given a genus-$g$ Heegaard splitting $(V, W; \Sigma)$ with $g \geq 2$ for $M$, a sphere $P$ embedded in $M$ is called a {\it Haken sphere} if $P \cap \Sigma$
is a single essential simple closed curve in $\Sigma$.
Two Haken spheres $P$ and $Q$ are said to be {\it equivalent} if $P \cap \Sigma$ is isotopic to $Q \cap \Sigma$ in $\Sigma$.
When the splitting $(V, W; \Sigma)$ admits Haken spheres, we denote by $\mu = \mu(V, W; \Sigma)$ the minimal cardinality of $P \cap Q \cap \Sigma$,
where $P$ and $Q$ vary over all pairwise non-equivalent Haken spheres for the splitting.
The {\it sphere complex} for the splitting $(V, W; \Sigma)$ is then defined to be the
simplicial complex whose vertices are equivalence classes of Haken spheres for the splitting
and whose $k$-simplices are collections of $k+1$ vertices represented by Haken spheres $P_0, P_1, \cdots, P_k$, respectively, such that the cardinality of $P_i \cap P_j \cap \Sigma$ is $\mu$ for all $0 \leq i < j \leq k$.

The structures of sphere complexes for genus-$2$ Heegaard splittings have been studied by several authors.
If a genus-$2$ Heegaard splitting for a $3$-manifold admits Haken spheres, then the manifold is one of $S^3$, $S^2 \times S^1$, lens spaces, and their connected sums.
It is known that the sphere complex for the genus-$2$ Heegaard splitting of $S^3$ is connected and even contractible from Scharlemann \cite{Sch04}, Akbas \cite{Akb08} and Cho \cite{Cho08}.
Lei \cite{Lei05} and Lei-Zhang \cite{LZ04} proved that the sphere complexes are connected for  genus-$2$ Heegaard splittings of non-prime 3-manifolds, that is, the connected sum whose summands are lens spaces or $S^2 \times S^1$, and later, in Cho-Koda \cite{CK13b} it is shown that they are actually contractible.


In Section \ref{sec:Sphere complexes}, we study the Heegaard splitting for a $3$-manifold having a single $S^2 \times S^1$ summand in its prime decomposition.
We prove that, if a genus-$g$ Heegaard splitting with $g \geq 2$ is the splitting obtained by
gluing a genus-$(g-1)$ Heegaard splitting of Hempel distance at least $2$ and
the genus-$1$ Heegaard splitting of $S^2 \times S^1$, then its sphere complex is a contractible, $(4g - 5)$-dimensional complex (Corollary \ref{cor:arc complex and sphere complex for the connected sums}).
In fact, we show that the sphere complex is isomorphic to the full subcomplex $\mathcal{A}^*_{g-1, 2}$ of the arc complex $\mathcal{A}_{g-1, 2}$.
As a special case, the sphere complex for the genus-$2$ Heegaard splitting of $S^2 \times S^1$ is a contractible, $3$-dimensional complex (Corollary \ref{cor:contractibility of sphere compex}).


For a Heegaard splitting $(V, W; \Sigma)$ for a $3$-manifold, the {\it Goeritz group} is defined to be the group of isotopy classes of the
orientation-preserving homeomorphisms of the manifold that preserve $V$ and $W$ setwise.
One might expect that the Goeritz group would be simpler once we have more complicated  Heegaard splitting in some sense.
One of the important results on Goeritz group in this view point is that the Goeritz groups of Heegaard splittings of Hempel distance at least $4$ are all finite groups, which is given in Johnson \cite{Joh10}.
On the other hand, it is hard to determine
whether the Goeritz group of a given Heegaard splitting of low Hempel distance is finitely generated or not.
Even it remains open weather the Goeritz group of a Heegaard splitting of genus at least 3
for the 3-sphere is finitely generated or not.
The Goeritz groups of genus-$1$ Heegaard splittings are easy to describe, and for genus-$2$ reducible Heegaard splittings,
their Goeritz groups have been studied in \cite{Goe33, Sch04, Akb08, Cho08, Cho12, CK12, CK13a, CK13b}.


In the final section, we study the Goeritz groups of the Heegaard splittings given in Section \ref{sec:Sphere complexes}.
The main result is that, for a Heegaard splitting obtained by gluing a Heegaard splitting of Hempel distance at least $4$ and the genus-$1$ Heegaard splitting of $S^2 \times S^1$,
then its Goeritz group is finitely generated (Corollary \ref{cor:generators of Goeritz groups for the connected sums}).
This can be compared with the result in Johnson \cite{Joh11} that, if a Heegaard splitting
is obtained by gluing a Heegaard splitting of high Hempel distance and the genus-$1$ Heegaard splitting of $S^3$,
then its Goeritz group is finitely generated.

\vspace{1em}
Throughout the paper, we will work in the smooth category unless otherwise mentioned.
By $\Nbd(X; Y)$ we will denote a regular neighborhood of a subspace $X$ of a polyhedral space $Y$.

\section{Arc complexes}
\label{sec:Arc complexes}

We start with recalling a sufficient condition for contractibility
of a simplicial complex, introduced in \cite{Cho08}, which is a generalization of the proof of Theorem 5.3 in \cite{McC91}.

Let $\mathcal{K}$ be a simplicial complex.
We call a vertex $w$ is {\it adjacent} to a vertex $v$ of $\mathcal K$ if $w$ equals $v$ or $w$ is joined to $v$ by an edge of $\mathcal K$.
We denote by $\st (v)$ the {\it star} of a vertex $v$ of $\mathcal K$ which is the full subcomplex of  $\mathcal{K}$
spanned by the vertices adjacent to $v$.
An {\it adjacency pair} $(X, v)$ in $\mathcal{K}$ is
a finite multiset that consists of vertices of $\st (v)$.
Here the finite multiset $X$ is a finite set $\{v_1, v_2, \ldots  , v_k\}$ allowed to have $v_i = v_j$ for some $1\leq i < j \leq k$.
A {\it remoteness function} for a vertex $v_0$ of $\mathcal K$ is a function
$r$ from the set of vertices of $\mathcal K$ to $\Natural \cup \{ 0 \}$ satisfying
$r^{-1} (0) \subset \st (v_0)$.
A {\it blocking function} for a remoteness function $r$ is a function $b$ from
the set of adjacency pairs of $\mathcal{K}$ to $\Natural \cup \{ 0 \}$ satisfying the following properties
for every adjacency pair $(X, v)$ with $r(v) > 0$:
\begin{enumerate}
\item
if $b (X, v) = 0$, then there exists a vertex $w$ of $\mathcal K$ joined to $v$ by an edge of $\mathcal{K}$ such that
$r (w) < r (v)$ and $(X, w)$ is also an adjacency pair (see Figure \ref{fig:contractibility} (a)).
\item
if $b (X, v) > 0$, then there exist an element $v'$ of $X$ and a vertex $w'$ of $\mathcal{K}$ which is joined to $v'$ by an edge of $\mathcal{K}$
such that
\begin{enumerate}
\item
$r (w') < r (v')$,
\item
if an element $x$ of $X$ is adjacent to $v'$, then $x$ is also adjacent to $w'$, and
\item
$b (X \setminus \{ v' \} \cup \{ w' \} , v) < b (X, v)$,
where $X \setminus \{ v' \} \cup \{ w' \}$ is the multiset
obtained by removing one instance of $v'$ from $X$ and
adding one instance of $w'$ to $X$ (see Figure \ref{fig:contractibility} (b)).
\end{enumerate}
\end{enumerate}

\smallskip

\begin{figure}[htbp]
\begin{center}
\includegraphics[width=10cm,clip]{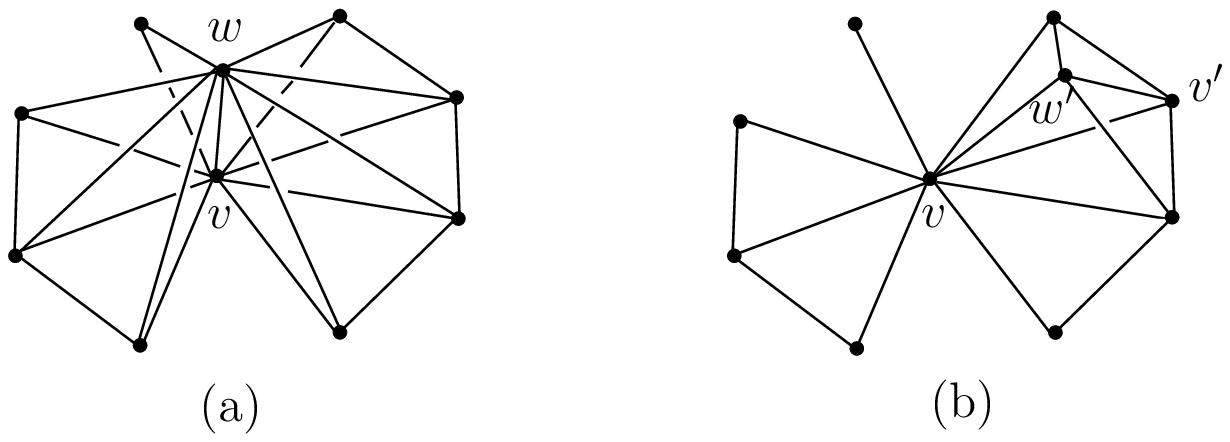}
\caption{}
\label{fig:contractibility}
\end{center}
\end{figure}


A simplicial complex $\mathcal{K}$ is called a {\it flag complex}
if any collection of pairwise distinct $k+1$ vertices span a
$k$-simplex whenever any two of
them span a 1-simplex.

\begin{lemma}[\cite{Cho08} Proposition 3.1]
\label{lem:sufficient condition for contractibility}
Let $\mathcal{K}$ be a flag complex with base vertex $v_0$.
If there exists a remoteness function $r$ on the set of vertices of $\mathcal{K}$ for $v_0$ that admits a blocking function $b$,
then $\mathcal{K}$ is contractible.
\end{lemma}

The idea of the proof given in \cite{Cho08} is to show that the homotopy groups are all trivial. That is, given any simplicial map $f: S^q \rightarrow \mathcal K$, $q \geq 0$, with respect to a triangulation $\Delta$ of $S^q$, we find a simplicial map $g: S^q \rightarrow \mathcal K$ with respect to a triangulation $\Delta'$ obtained from $\Delta$ by finitely many barycentric subdivisions, such that $g$ is homotopic to $f$ and the image of $g$ is contained in $\st(v_0)$.

\medskip

Now we return to the arc complex $\mathcal{A}_{g, n}$ of a compact orientable surface $\Sigma_{g, n}$ of genus $g$ with $n$ holes, where $n \geqslant 3$ if $g = 0$, and $n \geqslant 1$ if $g \geqslant 1$.
It is a standard fact that any collection of isotopy classes of essential arcs in $\Sigma_{g, n}$ can be realized by a collection of representative arcs having pairwise minimal intersection.
In particular, for a collection $\{v_0, v_1, \ldots , v_k\}$ of vertices of $\mathcal{A}_{g, n}$ if $v_i$ and $v_j$ are joined by an edge for each $0 \leq i < j \leq k$, then $\{v_0, v_1, \ldots , v_k\}$ spans a $k$-simplex.
Thus we have the following.

\begin{lemma}
\label{lem:arc complex is a flag complex}
The arc complex $\mathcal{A}_{g, n}$ is a flag complex, and any full subcomplex of $\mathcal{A}_{g,n}$ is
also a flag complex.
\end{lemma}

Let $\alpha$ and $\alpha_0$ be essential arcs on the surface $\Sigma_{g, n}$
which intersect each other transversely and minimally.
A component $\beta$ of $\alpha_0$ cut off by $\alpha \cap \alpha_0$ is said to be {\it outermost}
if $\beta \cap \alpha$ consists of a single point.
We note that there exists exactly two such subarcs of $\alpha_0$.
The intersection $\beta \cap \alpha$ cuts $\alpha$ into two subarcs $\beta'$ and $\beta''$.
We call the two new arcs $\alpha' = \beta \cup \beta'$ and $\alpha'' = \beta \cup \beta''$ the
{\it arcs obtained from $\alpha$ by surgery along $\beta$.}
We observe that by a small isotopy $\alpha'$ and $\alpha''$ are disjoint from $\alpha$, and $|\alpha_0 \cap \alpha'| < |\alpha_0 \cap \alpha|$ and $|\alpha_0 \cap \alpha''| < |\alpha_0 \cap \alpha|$ since the intersection $\beta \cap \alpha$ is no longer counted.

\begin{theorem}
\label{thm:contractibility of arc complexes}
Any full subcomplex $\mathcal{A}$ of $\mathcal{A}_{g, n}$ satisfying the following property
is contractible.
\begin{description}
\item[Surgery Property]
Let $\alpha$ and $\alpha_0$ be representative arcs of vertices of $\mathcal{A}$ that intersect each other
transversely an minimally.
If $\alpha \cap \alpha_0 \neq \emptyset$,
then at least one of the two arcs obtained from $\alpha$ by surgery along an outermost subarc of $\alpha_0$
cut off by $\alpha \cap \alpha_0$ represents a vertex of $\mathcal{A}$.
\end{description}
\end{theorem}

\begin{proof}
Fix a base vertex $v_0$ of $\mathcal{A}$.
By Lemmas
\ref{lem:sufficient condition for contractibility}
and \ref{lem:arc complex is a flag complex}, it suffices to find
a remoteness function for $v_0$ that admits
a blocking function.
For each vertex $v$ of $\mathcal A$, define $r(v)$ to be the minimal cardinality of the intersection $\alpha \cap \alpha_0$,
where $\alpha$ and $\alpha_0$ are representative arcs of
$v$ and $v_0$, respectively.
By definition, $r$ is a remoteness function for $v_0$.

Let $(X, v)$ be an adjacent pair in $\mathcal{A}$, where
$r(v) > 0$ and $X = \{ v_1, v_2, \ldots, v_n \}$.
Choose representative arcs $\alpha$, $\alpha_1$, $\alpha_2 , \ldots, \alpha_n$ and $\alpha_0$
of $v$, $v_1$, $v_2 , \ldots, v_n$ and $v_0$, respectively, so that they have transversal and pairwise minimal intersection, and every crossing is a double point.
Since $r(v) > 0$, we have $\alpha \cap \alpha_0 \neq \emptyset$.
Among the two subarcs of $\alpha_0$ cut off by $\alpha \cap \alpha_0$, choose one, say $\beta$, so that the cardinality of $\beta \cap (\alpha_1 \cup \alpha_2 \cup \cdots \cup \alpha_n)$ is minimal, and then denote this cardinality by $b_0 = b_0 (\alpha $, $\alpha_1 $, $\alpha_2 , \ldots, \alpha_n $, $\alpha_0 )$.
We define $b (X , v)$ to be the minimal number
of $b_0$ among all such representative arcs of
$v$, $v_1$, $v_2, \ldots, v_n$ and $v_0$.
In the following, we will show that $b$ is a blocking function for $r$.

First, suppose that $b (X, v) = 0$.
Then by an isotopy we may assume that $\beta \cap (\alpha_1 \cup \alpha_2 \cup \cdots \cup \alpha_n) = \emptyset$.
By the Surgery Property, at least one of the two arcs
obtained from $\alpha$ by surgery along $\beta$, say $\alpha'$, represents a vertex
$w$ of $\mathcal{A}$.
By the construction, $v$ is adjacent to $w$, and $(X, w)$ is an adjacent pair.
Further, we have
$r (w) \leqslant | \alpha_0 \cap \alpha' | < | \alpha_0 \cap \alpha | = r (v)$.

Next, suppose that $b (X, v) > 0$.
We may assume that $\beta \cap (\alpha_1 \cup \alpha_2 \cup \cdots \cup \alpha_n) = b (X, v)$ by an isotopy.
Let $\gamma$ be the outermost subarc of $\alpha$ cut off by $\alpha_1 \cup \alpha_2 \cup \cdots \cup \alpha_n$
that is contained in $\beta$.
The point $(\alpha_1 \cup \alpha_2 \cup \cdots \cup \alpha_n) \cap \gamma$ is contained in $\alpha_k$ for some
$k \in \{ 1 , 2, \ldots, n \}$.
Then by the Surgery Property again, at least one of the arcs
obtained from $\alpha_k$ by surgery along $\gamma$ represents a vertex, say $w'$, of $\mathcal{A}$.
By the construction, we  have
$r (w') < r (v_k)$, $b (X \setminus \{ v_k\} \cup \{ w' \} , v) < b (X, v)$, and
each element $x$ of $X$ adjacent to $v_k$ is also adjacent to $w'$.
This completes the proof.
\end{proof}

Let $n \geqslant 2$.
We denote by $\mathcal{A}^*_{g, n}$ the full subcomplex of $\mathcal{A}_{g, n}$
spanned by the vertices represented by simple arcs connecting
the different components of the boundary of $\Sigma_{g, n}$.
It is easy to verify that the arc complex $\mathcal{A}_{g, n}$ itself and the subcomplex $\mathcal{A}^*_{g, n}$ satisfy the Surgery Property.
Thus we have the following.

\begin{corollary}
\label{cor:subcomplex of arc complex}
The complexes $\mathcal{A}_{g, n}$ and $\mathcal{A}^*_{g, n}$ are  contractible.
\end{corollary}

We end the section with the following lemma for later use.

\begin{lemma}
\label{lem:dimensions of arc complexes}
The dimension of the complex $\mathcal{A}^*_{g, 2}$ is $4g - 1$.
\end{lemma}
\begin{proof}
Let $A = \{ \alpha_1 , \alpha_2, \ldots, \alpha_n \}$ be
a maximal set of mutually disjoint, mutually non-isotopic simple arcs connecting
the different components of the boundary of $\Sigma_{g, 2}$.
By contracting each of these boundary components of $\Sigma_{g, 2}$ into a point,
we get a closed orientable surface $\Sigma$ of genus $g$ with 2 dots, say $v^+$ and $v^-$.
On this surface, each of the arcs of $A$ connects the two dots.
Hence $A$ decomposes $\Sigma$ into cubes with the vertex sets $\{ v^+ , v^- \}$.
Now, the assertion follows easily from Euler characteristic considerations.
\end{proof}

\section{Sphere complexes}
\label{sec:Sphere complexes}

Let $(V, W; \Sigma)$ be
a Heegaard splitting of genus $g \geq 2$ of a closed orientable 3-manifold $M$.
A separating sphere $P$ embedded in $M$ is called a {\it Haken sphere} for the spitting if
it intersects $\Sigma$ transversely
in a single essential simple closed curve.
Since $P$ is separating in $M$, the curve $P \cap \Sigma$
is separating in $\Sigma$.
Two Haken spheres $P$ and $Q$ are said to be {\it equivalent} if
$P \cap \Sigma$  and $Q \cap \Sigma$ are isotopic in $\Sigma$.
We denote by $\mu = \mu (V, W; \Sigma)$ the minimal cardinality of
$P \cap Q \cap \Sigma$,
where $P$ and $Q$ vary over all pairwise non-equivalent Haken spheres for $(V, W; \Sigma)$.
We note that $\mu$ is a non-negative even number.
It was shown in \cite{ST03} that $\mu (V, W; \Sigma) = 4$ when
the genus of the splitting is 2.
When the given splitting $(V, W; \Sigma)$ admits Haken spheres, the {\it sphere complex} for the splitting is defined as in Introduction, which we will denote by $\mathcal H = \mathcal H(V, W; \Sigma)$.

Given a closed orientable surface $\Sigma$ of genus $g \geq 1$, the {\it curve complex}  $\mathcal C_g$ is defined to be the simplicial complex whose vertices are isotopy classes of simple closed curves in $\Sigma$ and whose $k$-simplices are collections of $k+1$ vertices  represented by pairwise disjoint and non-isotopic curves in $\Sigma$.
It is known that the curve complex $\mathcal C_g$ is connected and $(3g-4)$-dimensional.
When the surface is the Heegaard surface of a Heegaard splitting $(V, W; \Sigma)$ of a 3-manifold,
we have the two full subcomplexes $\mathcal D_V$ and $\mathcal D_W$ of $\mathcal C_g$ which are spanned by the vertices of the simple closed curves bounding disks in $V$ and $W$, respectively.
Then we define the {\it Hempel distance} of the splitting to be the minimal simplicial distance in $\mathcal C_g$ between the two subcomplexes $\mathcal D_V$ and $\mathcal D_W$.
That is, the minimal number of edges among all the paths in $\mathcal C_g$ from a vertex of $\mathcal D_V$ to a vertex of $\mathcal D_W$.
We refer \cite{Hem01} to the reader for details on the Hempel distance.
In the case of genus-1 Heegaard splitting for a 3-manifold, we have the Hempel distance 0 if the manifold is $S^2 \times S^1$, and the distance is $\infty$ otherwise.

Let $(V, W; \Sigma)$ be a Heegaard splitting of a closed orientable 3-manifold $M$.
A non-separating disk $E_0$ in $V$ is called a {\it reducing disk}
if $\partial E_0$ bounds a disk in $W$.
We note that if there exists a reducing disk in $M$, then $M$ has
an $S^2 \times S^1$ summand for its prime decomposition, and vice versa by 
Waldhausen's uniqueness of Heegaard splittings of $S^2 \times S^1$ \cite{Wal68} and Haken's lemma \cite{Hak68}.
Given any simple closed curve $\gamma$ in $\Sigma$ intersecting $\partial E_0$ transversely in a single point, the boundary of $\Nbd(\partial E_0 \cup \gamma; \Sigma)$ is a separating simple closed curve in $\Sigma$ which bound a disk in each of $V$ and $W$.
Thus, if the genus of the splitting is greater than 1, such a simple closed curve $\gamma$ determines a Haken sphere $P = P(\gamma)$ for the splitting, the union of those two disks in $V$ and $W$.

\begin{lemma}
\label{lem:transitivity of Goeritz group on the set of Haken spheres}
Let $(V, W; \Sigma)$ be a genus-$g$ Heegaard splitting of a $3$-manifold $M$, where $g \geqslant 2$.
Let $E_0$ be a reducing disk in $V$.
Let $\gamma$ and $\gamma'$ be simple closed curves each of which intersects
$\partial E_0$ transversely in a single point.
Let $P= P(\gamma)$ and $P'=P'(\gamma')$ be Haken spheres determined by the curves $\gamma$ and $\gamma'$, respectively.
Then there exists an orientation-preserving homeomorphism of the manifold $M$ onto itself
that maps $P$ to $P'$ while preserving each of $V$ and $W$ setwise.
\end{lemma}
\begin{proof}
If $\gamma'$ is isotopic to $\gamma$ up to Dehn twists about $\partial E_0$, 
$P$ and $P'$ are equivalant, thus there is nothing to prove. 
Otherwise, suppose first that $\gamma'$ is disjoint from $\gamma$ up to 
Dehn twists about $\partial E_0$. 
We may assume without loss of generality that $\gamma'$ itself is disjoint from $\gamma$ because 
Dehn twists about $\partial E_0$ does not change the equivalence class of $P'$. 
The boundary of $\Nbd (\partial E_0 ; \Sigma)$ consists of two simple closed curves 
$\delta_1$ and $\delta_2$. 
For each $i \in \{1, 2\}$, 
the intersection of $\delta_i$ and $\gamma_1 \cup \gamma_2$ cuts $\delta_i$ into two arcs 
$\delta_{i,1}$ and $\delta_{i,2}$.  
We set $\alpha_i = ((\gamma_1 \cup \gamma_2) \setminus \Nbd (\partial E_0 ; \Sigma)) 
\cup \delta_{1,1} \cup \delta_{2, i}$. 
We note the union $\alpha_1 \cup \alpha_2$ bounds proper annuli $A$ and $B$ 
in $V$ and $W$, respectively. 
Then a single Dehn twist about the annulus $A$, which extends to the Dehn twist about 
the torus $A \cup B$, is the required homeomorphism of $M$. 
Here we remark that this homeomorphism is actually a ``sliding" of a foot of the $1$-handle of 
each of $V$ and $W$ whose belt sphere is $\partial E_0$.  

The general case follows from the connectivity of the arc complex as follows. 
The simple closed curve $\partial E_0$ cuts $\Sigma$ into a genus-$(g-1)$ surface $\Sigma_{g-1, 2}$
with 2 holes $\partial E_0^+$ and $\partial E_0^-$ coming from $\partial E_0$.
On the surface $\Sigma_{g-1, 2}$, $\gamma$ and $\gamma'$ are simple arcs 
$\beta$ and $\beta'$ connecting the two holes.
Since the complex $\mathcal A^*_{g-1, 2}$ is connected by Corollary \ref{cor:subcomplex of arc complex}, 
there exists a sequence $\beta = \beta_1, \beta_2, \ldots, \beta_n = \beta'$ of mutually non-isotopic, 
essential arcs in $\Sigma_{g-1, 2}$ connecting $\partial E_0^+$ and $\partial E_0^-$ such that $\beta_i$ is disjoint from $\beta_{i+1}$ for each $i \in \{1, 2, \ldots , n-1\}$. 
Gluing $\partial E_0^+$ and $\partial E_0^-$ back, this sequence gives rise to 
a sequence of 
simple closed curves $\gamma = \gamma_1, \gamma_2, \ldots, \gamma_n = \gamma'$ such that 
$\gamma_{i+1}$ is disjoint from $\gamma_{i}$ up to 
Dehn twists about $\partial E_0$ for each $i \in \{1, 2, \ldots , n-1\}$. 
Then there exists an orientation-preserving homeomorphism $g_i$ of the manifold $M$ onto itself
that maps $P(\gamma_i)$ to $P(\gamma_{i+1})$ while preserving each of $V$ and $W$ setwise. 
Then the composition $g_{n-1}g_{i-2}\cdots g_1$ is the desired homeomorphism.
\end{proof}

Let $(V, W; \Sigma)$ be a genus-$g$ Heegaard splitting of a 3-manifold $M$ with $g \geq 2$.
Suppose that there exists a reducing disk $E_0$ in $V$. 
We denote by $\mathcal{H}_{E_0}$ the simplicial complex
whose vertices are equivalence classes of
Haken spheres $P = P(\gamma)$ determined by simple closed curves $\gamma$ in $\Sigma$ intersecting $\partial E_0$ transversely in a single point, and whose $k$-simplices are collections of $k+1$ vertices represented by pairwise non-equivalent Haken spheres $P(\gamma_0), P(\gamma_1), \ldots, P(\gamma_k)$ such that the minimal cardinality of each $P(\gamma_i) \cap P(\gamma_j) \cap \Sigma$ is $4$, for $0 \leq i < j \leq k$ (equivalently the arcs $\gamma_i$ and $\gamma_j$ are disjoint from each other).
We observe that a Haken sphere $P$ represents a vertex of $\mathcal{H}_{E_0}$ if and only if $P$ cuts off from $V$ a solid torus whose meridian disk is $E_0$.
By construction, if $\mu (V, W; \Sigma) = 4$, then the complex $\mathcal{H}_{E_0}$ is a full subcomplex of the sphere complex $\mathcal{H}$ of the splitting $(V, W; \Sigma)$.
The following lemma is immediate from the definition of the complex $\mathcal H_{E_0}$ with Corollary \ref{cor:subcomplex of arc complex} and Lemma \ref{lem:dimensions of arc complexes}.

\begin{lemma}
\label{lem:contractibility of sphere compex}
Let $(V, W; \Sigma)$ be a genus-$g$ Heegaard splitting of a closed orientable $3$-manifold $M$ with $g \geqslant 2$.
Let $E_0$ be a reducing disk in $V$.
Then the complex $\mathcal{H}_{E_0}$
is isomorphic to the complex $\mathcal{A}^*_{g-1, 2}$, and hence it is a contractible, $(4g - 5)$-dimensional complex.
\end{lemma}

\begin{proposition}
\label{prop:H and HE0}
Let $(V, W; \Sigma)$ be a genus-$g$ Heegaard splitting of a closed orientable $3$-manifold $M$ 
with $g \geqslant 2$.
Suppose that there exists a unique reducing disk $E_0$ in $V$, and also that $\mu(V, W; \Sigma) > 0$.
Then the sphere complex $\mathcal{H}$ for the splitting $(V, W ; \Sigma)$ coincides with the complex $\mathcal{H}_{E_0}$.
\end{proposition}
\begin{proof}
Let $P$ be a Haken sphere for the splitting $(V, W; \Sigma)$ intersecting $E_0$ transversely and minimally.
Suppose that $P \cap E_0 \neq \emptyset$. 
At least one, say $M_1$, of the closed 3-manifolds $M_1$ and $M_2$ obtained by cutting $M$ along $P$ 
and then capping off the resulting boundary spheres by adding 3-balls has $S^2 \times S^1$ summand for its 
prime decomposition. 
Then as mentioned in the last paragraph before Lemma \ref{lem:transitivity of Goeritz group on the set of Haken spheres}, 
the $V$ part of the Heegaard splitting of $M_1$ naturally induced from $(V, W; \Sigma)$ 
contains a reducing disk, 
which gives rise to a reducing disk of $V$ that not isotopic to $E_0$. 
This contradicts the uniqueness of $E_0$. 
Therefore any Haken sphere is disjoint from the reducing disk $E_0$.
It suffices to show that $P$ cuts off a solid torus from $V$ whose meridian disks is $E_0$.
Suppose not.
That is, the component $\Sigma'$ of $\Sigma$ cut off by $P \cap \Sigma$ containing $\partial E_0$
is a compact surface of genus at least $2$.
Then we can choose a simple closed curve $\gamma$ in $\Sigma'$ intersecting
$\partial E_0$ transversely in a single point such that the Haken sphere $Q = Q(\gamma)$ is disjoint from and is not equivalent to $P$.
We have then $0 < \mu(V, W; \Sigma) \leq |P \cap Q \cap \Sigma| = 0$, a contradiction.
\end{proof}


Now we will construct (infinitely many) Heegaard splittings $(V, W; \Sigma)$
satisfying the conditions in Proposition \ref{prop:H and HE0}:
\begin{itemize}
\item
there exists a unique reducing disk in $V$; and
\item
$\mu (V, W; \Sigma) > 0$.
\end{itemize}

Let $(V_1, W_1; \Sigma_1)$ and  $(V_2, W_2; \Sigma_2)$ be genus-$g_1$ and genus-$g_2$ Heegaard splittings for  $3$-manifolds $M_1$ and $M_2$, respectively.
Let $B_1$ and $B_2$ be 3-balls in $M_1$ and $M_2$  which intersect the Heegaard surfaces $\Sigma_1$ and $\Sigma_2$ in a single disk, respectively.
Removing the interiors of $B_1$ and $B_2$, and identifying $\partial B_1$ and $\partial B_2$, we can construct a genus-$(g_1 + g_2)$ Heegaard splitting $(V, W ; \Sigma)$ for the connected sum $M = M_1 \# M_2$ such that $V$ and
$W$ are considered as boundary connected sums of $V_1$ and $V_2$, and $W_1$ and $W_2$,  respectively.
We call the splitting $(V, W ; \Sigma)$ a Heegaard splitting for $M$ {\it obtained from $(V_1, W_1; \Sigma_1)$ and $(V_2, W_2; \Sigma_2)$}.
We note that the sphere $P = \partial B_1 = \partial B_2$ is a Haken sphere for the splitting $(V, W ; \Sigma)$.
In the remaining of the section, we always assume the following:

\begin{itemize}
\item
$(V_1, W_1; \Sigma_1)$ is a genus-$(g-1)$ Heegaard splitting for a closed orientable 3-manifold $M_1$,
with $g \geqslant 2$, and $(V_2, W_2; \Sigma_2)$ is the genus-$1$ Heegaard splitting for $S^2 \times S^1$.
\item
$(V, W ; \Sigma)$ is a genus-$g$ Heegaard splitting for $M = M_1 \# (S^2 \times S^1)$
obtained from $(V_1, W_1; \Sigma_1)$ and $(V_2, W_2; \Sigma_2)$ by the above construction, and $P = \partial B_1 = \partial B_2$
is the Haken sphere for the splitting $(V, W ; \Sigma)$.
\item
$E_0$ and $E'_0$ with $\partial E_0 = \partial E'_0$ are meridian disks of the solid tori $V_2$ and $W_2$ respectively, which are reducing disks for the splitting $(V, W; \Sigma)$.
\end{itemize}

We start with the following two lemmas.

\begin{lemma}
\label{lem:non-separating disk disjoint from a reducing disk}
Let $\delta$ be an essential simple closed curve in $\Sigma$ that is disjoint from and not isotopic to $\partial E_0$.
If $\delta$ bounds disks in $V$ and $W$ simultaneously, then the Hempel distance of the splitting $(V_1, W_1; \Sigma)$ is $0$.
\end{lemma}

From the lemma, it is easy to see that, if the Hempel distance of the splitting $(V_1, W_1; \Sigma_1)$ is at least $1$ and if $E$ is an essential non-separating disk in $V$ that is disjoint from and not isotopic to $E_0$, then $E$ cannot be a reducing disk for the splitting $(V, W; \Sigma)$.

\begin{proof}[Proof of Lemma $\ref{lem:non-separating disk disjoint from a reducing disk}$]
Suppose that $\delta$ in $\Sigma$ bounds disks both in $V$ and $W$.
We want to find an essential simple closed curve in $\Sigma_1$ which bounds disks both in $V_1$ and $W_1$.

Among the simple closed curves in $\Sigma$ that intersect $\partial E_0$ transversely in a single point, choose one, say $\gamma$, so that $\gamma$ intersects $\delta$ minimally.
Then we have that either $\gamma$ is disjoint from $\delta$ or $\gamma$ intersects $\delta$ in a single point.
(If $\delta$ is non-separating and $\delta \cup \partial E_0$ is separating in $\Sigma$, then we have to choose such a curve $\gamma$ so that $\gamma$ intersects $\delta$ in a single point. Otherwise, we can choose $\gamma$ disjoint from $\delta$.)
Let $P(\gamma)$ be the Haken sphere determined by $\gamma$.
That is $P(\gamma) \cap \Sigma$ is the boundary of $\Nbd(\partial E_0 \cup \gamma; \Sigma)$.
Applying Lemma \ref{lem:transitivity of Goeritz group on the set of Haken spheres},
we may assume that the Haken sphere $P (= \partial B_1 = \partial B_2)$ equals $P(\gamma)$, and that $\Nbd(E_0 \cup \gamma; V)$ and $\Nbd(E'_0 \cup \gamma; W)$ are solid tori $V_2$ and $W_2$, respectively, with the interior of the $3$-ball $B_2$ removed.

If $\gamma$ is disjoint from $\delta$, then by isotopy we may assume that $\delta$ lies in $\Sigma_1$ outside the disk $B_1 \cap \Sigma_1$, and that the two disks in $V$ and $W$ bounded by $\delta$ are disjoint from $P$.
Apparently, $\delta$ remains to be essential in $\Sigma_1$.
Thus the Hempel distance of $(V_1, W_1; \Sigma_1)$ is $0$ in this case.
If $\gamma$ intersects $\delta$ in a single point, then we cannot say that $\delta$ lies in $\Sigma_1$.
But by isotopy we may assume that the boundary of $\Nbd(\partial E_0 \cup \delta \cup \gamma; \Sigma)$, which consists of two simple closed curves, lies in $\Sigma_1$ outside the disk $B_1 \cap \Sigma_1$.
Any of the two simple closed curves bound disks in $V$ and $W$, which can be isotoped to be disjoint from $P$.
Since $\delta$ is not isotopic to $\partial E_0$ in $\Sigma$, each of these simple closed curves is essential in $\Sigma_1$.
Again the Hempel distance of $(V_1, W_1; \Sigma)$ is $0$.
\end{proof}
	
\begin{lemma}
\label{lem:disk and a simple closed curve}
Let $\delta_1$ and $\delta_2$ be disjoint, essential simple closed curves in $\Sigma$ each of which is disjoint from and not isotopic to $\partial E_0$.
If $\delta_1$ bounds a disk in $V$ and $\delta_2$ bounds a disk in $W$, then the Hempel distance of the splitting $(V_1, W_1; \Sigma)$ is at most $1$.
\end{lemma}

\begin{proof}
The argument will be very similar to the proof of Lemma \ref{lem:non-separating disk disjoint from a reducing disk}.
We note that $\delta_1$ is possibly isotopic to $\delta_2$.
Suppose that $\delta_1$ and $\delta_2$ bound disks in $V$ and $W$ respectively.
We want to find two disjoint, essential simple closed curves in $\Sigma_1$ such that one bound a disk in $V_1$
and the other in $W_1$.
Among the simple closed curves in $\Sigma$ that intersect $\partial E_0$ transversely in a single point, choose one, say $\gamma$, so that $\gamma$ intersects
$\delta_1 \cup \delta_2$ minimally.
We may assume that the Haken sphere $P$ equals $P(\gamma)$ as in the proof of Lemma \ref{lem:non-separating disk disjoint from a reducing disk}.
For each $i \in \{1, 2\}$, $\delta_i$ is disjoint from $\gamma$ or intersects $\gamma$ in a single point, and hence we have four cases.

If each of $\delta_1$ and $\delta_2$ is disjoint from $\gamma$, then by isotopy we may assume that $\delta_1$ and $\delta_2$ lie inside $\Sigma_1$
as disjoint, essential simple closed curves, and these bound disks inside $V_1$ and $W_1$ respectively.
If one of them, say $\delta_1$, intersects $\gamma$ in a single point and the other one $\delta_2$ is disjoint from $\gamma$, then consider the boundary of $\Nbd(\partial E_0 \cup \delta_1 \cup \gamma; \Sigma)$, which consists of two simple closed curves.
By isotopy we may assume that both of the two simple closed curves lie inside $\Sigma_1$ as essential simple closed curves, and bound disks in $V_1$,
while $\delta_2$ is an essential simple closed curve in $\Sigma_1$ disjoint from these curves and bounding a disk in $W_1$.
Finally, if each of $\delta_1$ and $\delta_2$ intersects $\gamma$ in a single point, then consider the boundary of $\Nbd(\partial E_0 \cup \delta_1 \cup \delta_2 \cup \gamma; \Sigma)$, which consists of three simple closed curves.
By isotopy again, we may assume that all the three curves lie in $\Sigma_1$.
Among the three curves, one is a component of the boundary of $\Nbd(\partial E_0 \cup \delta_1 \cup \gamma; \Sigma)$ which bounds a disk in $V_1$,
and another one is a component of the boundary of $\Nbd(\partial E_0 \cup \delta_2 \cup \gamma; \Sigma)$ which bounds a disk in $W_1$.
(The third one may bound a disk neither in $V_1$ nor in $W_1$.)
Again, these two simple closed curves are essential in $\Sigma_1$.
Therefore, in any of four cases, the Hempel distance of the splitting $(V_1, W_1; \Sigma)$ is at most $1$.
\end{proof}

Let $D$ and $E$ be essential disks in the handlebody $V$
which intersect each other transversely and minimally.
A subdisk $\Delta$ of $D$ cut off by $D \cap E$ is said to be {\it outermost}
if $\Delta \cap E$ is a single arc.
For an outermost subdisk $\Delta$ of $D$ cut off by $D \cap E$, the arc $\Delta \cap E$ cuts $E$ into two disks, say $E'$ and $E''$.
We call the two disks $E_1 = E' \cup \Delta$ and $E_2 = E'' \cup \Delta$ the
{\it disks obtained from $E$ by surgery along $\Delta$.}
Both of $E_1$ and $E_2$ can be isotoped to be disjoint from $E$.
By an elementary argument of the reduced homology group $H_2 (V , \partial V; \Integer)$, we can check easily that
at least one of $E_1$ and $E_2$ is non-separating if $E$ is non-separating.

For any simple closed curves $\gamma$ and $\delta$ in the surface $\Sigma$ which intersect each other transversely and minimally
in at least two points, we can define similarly the two simple closed curves $\gamma_1$ and $\gamma_2$ obtained from $\gamma$ by surgery along an outermost subarc of $\delta$ cut off by $\gamma \cap \delta$.
Here an outermost subarc, say $\delta'$, is a component of $\delta$ cut off by $\gamma \cap \delta$ which meets $\gamma$ only in its endpoints and cuts $\gamma$ into two arcs, say $\gamma'$ and $\gamma''$.
Then $\gamma_1 = \gamma' \cup \delta'$ and $\gamma_2 = \gamma'' \cup \delta'$.
If the subarc $\delta'$ meets $\gamma$ in the same side, 
then both of $\gamma_1$ and $\gamma_2$ can be isotoped to be disjoint from $\gamma$.
We also see that, if $\gamma$ is non-separating, then at least one of $\gamma_1$ and $\gamma_2$ are non-separating, by an elementary argument of $H_2 (\Sigma; \Integer)$.

\begin{proposition}
\label{prop:uniqueness of the reducing disk}
Suppose that the Hempel distance of the splitting $(V_1, W_1; \Sigma)$ is at least $2$.
Then we have the following:
\begin{enumerate}
\item
there exists a unique reducing disk in $V$, and
\item
$\mu (V, W; \Sigma) > 0$.
\end{enumerate}
\end{proposition}
\begin{proof}
(2) is easy to verify. In fact, if $\mu (V, W; \Sigma) = 0$, then one might find a Haken sphere for the splitting $(V_1, W_1; \Sigma_1)$, and hence the Hempel distance of $(V_1, W_1; \Sigma_1)$ is $0$, a contradiction.

In the following, we prove (1).
Let $E$ be an essential non-separating disk in $V$ that is not isotopic to $E_0$.
We may assume that $E$ intersects $E_0$ transversely and minimally.
If $E$ is disjoint from $E_0$, then $E$ is not a reducing disk by
Lemma \ref{lem:non-separating disk disjoint from a reducing disk}.
Suppose that $E$ intersects $E_0$.
Then $\partial E_0$ cuts $\partial E$ into $2n$ ($n \geqslant 1$) simple arcs
$\delta_1$, $\delta_2, \ldots, \delta_{2n}$.
We divide the collection of these arcs into two subcollections as
\[
\{ \delta_1 , \delta_2 , \ldots, \delta_{2n} \} =
\{ \delta_{1, 1} , \delta_{1, 2} , \ldots, \delta_{1, n_1} \} \sqcup
\{ \delta_{2, 1} , \delta_{2, 2} , \ldots, \delta_{2, n_2} \} ,
\]
where each of the arcs $\delta_{1, i}$ meets $\partial E_0$ in the same side while each of $\delta_{2, j}$ in the opposite sides.
We may assume without loss of generality that there exists
an outermost subdisk $\Delta$ of $E$ cut off by $E \cap E_0$ such that
$\delta_{1, 1} \subset \partial \Delta$.
Let $\{E_0', D_1', D_2', \ldots, D_{g-1}'\}$ be a complete system of meridian disks of $W$,
where $\partial E_0' = \partial E_0$.
Fix orientations of the boundary circles $\partial E_0'$ and $\partial D_1', \partial D_2', \ldots, \partial D_{g-1}'$, and assign symbols $x$ and $y_1, y_2, 
\ldots , y_{g-1}$ on the circles, respectively.
Then any oriented simple closed curve $\delta$ in $\Sigma$ intersecting the boundary circles transversely determines a
word $w(\delta )$ on $\{x, y_1 , y_2 , \ldots , y_{g-1}\}$
that can be read off from the intersections of $\delta$ with the circles.
This word determines an element of the free group $\pi_1 (W) = \langle x, y_1 , y_2 , \ldots , y_{g-1} \rangle $
represented by $\delta$.
Let $E_1$ and $E_2$ be the disks obtained from $E_0$ by surgery along $\Delta$.
Since $E_0$ is non-separating, at least one of the two, say $E_1$, is non-separating.
By a small isotopy, we assume that $E_1$ is disjoint from $E_0$.
\begin{claim}
\label{claim:subword read off by delta11}
The word $w (\delta_{1,1})$ on $\{ y_1 , y_2 , \ldots , y_{g-1} \}$ read off by the interior of the arc $\delta_{1, 1}$ represents a non-trivial element of $\pi_1 (W) = \langle x , y_1 , y_2 , \ldots , y_{g-1} \rangle$.
\end{claim}
\noindent {\it Proof of Claim $1$}.
The disk $E_1$ is non-separating, disjoint from $E_0$ and not isotopic to $E_0$, and hence, by Lemma \ref{lem:non-separating disk disjoint from a reducing disk}, it is not a reducing disk.
That is, $\partial E_1$ does not bound a disk in $W$.
Thus by the Loop Theorem $w (\partial E_1) = w (\delta_{1, 1})$ determines a  non-trivial element of $\pi_1 (W)$.

\begin{claim}
\label{claim:subword read off by delta1i}
The word $w (\delta_{1,i})$ on $\{ y_1 , y_2 , \ldots , y_{g-1} \}$ read off by the interior of the arc $\delta_{1, i}$ $(2 \leqslant i \leqslant n_1)$ represents a non-trivial element of $\pi_1 (W) = \langle x , y_1 , y_2 , \ldots , y_{g-1} \rangle$.
\end{claim}
\noindent {\it Proof of Claim $2$}.
The arc $\delta_{1, i}$ is an outermost subarc of $\partial E$ cut off by $\partial E \cap \partial E_0$.
One of the two simple closed curves obtained from $\partial E_0$ by surgery along $\delta_{1, i}$ is a non-separating curve, which we denote by $\gamma_0$.
By a small isotopy, we may assume that $\gamma_0$ is disjoint from $\partial E_0$.
Further, we observe that
$\gamma_0$ intersects $\partial E_1$ transversely at most once.
If $\gamma_0$ is disjoint from $\partial E_1$, then $\gamma_0$ cannot bound a disk in $W$ by Lemma \ref{lem:disk and a simple closed curve} since we assumed that the Hempel distance of the splitting $(V_1, W_1; \Sigma)$ is at least $2$.
If $\gamma_0$ intersects $\partial E_1$ in a single point, then the boundary of $\Nbd(\gamma_0 \cup \partial E_1; \Sigma)$ is a simple closed curve, which is disjoint from $\gamma_0$ and bounds a disk in $V$.
Thus $\gamma_0$ cannot bound a disk in $W$ again by Lemma \ref{lem:disk and a simple closed curve}.
Hence by the Loop Theorem the word
$w (\gamma_0) = w (\delta_{1, i})$ determines a non-trivial element of $\pi_1 (W)$.

\medskip

Now we can write the word $w (\partial E )$ on $\{x , y_1 , y_2 , \ldots , y_{g-1}\}$
as
$$x^{\epsilon_1} w(\delta_1) x^{\epsilon_2} w(\delta_2) x^{\epsilon_3}w(\delta_3) \cdots x^{\epsilon_{2n}}w (\delta_{2n}),$$
where $\epsilon_k \in \{-1, 1\}$ for $k \in \{1, 2, \ldots , 2n\}$.
If $\delta_k = \delta_{1, i}$ for some $i \in \{1, 2,  \ldots , n_1\}$, then $\{\epsilon_k, \epsilon_{k+1}\} = \{-1, 1\}$ but $w(\delta_k)$ is non-trivial.
If $\delta_k = \delta_{2, j}$ for some $j \in \{1, 2,  \ldots , n_2\}$, then $w(\delta_k)$ is possibly trivial but $\epsilon_k = \epsilon_{k+1}$. (Here $\epsilon_{2n+1} = \epsilon_1$.)
This implies that $w (\partial E )$ determines a non-trivial element of $\pi_1 (W)$, and so $\partial E$ cannot bound a disk in $W$.
Thus $E$ cannot be a reducing disk.
\end{proof}

By Lemma \ref{lem:contractibility of sphere compex}, and Propositions
\ref{prop:H and HE0} and
\ref{prop:uniqueness of the reducing disk},
we have the main result of the section:
\begin{corollary}
\label{cor:arc complex and sphere complex for the connected sums}
Let $(V_1, W_1; \Sigma_1)$ be a genus-$(g-1)$ Heegaard splitting of Hempel distance at least $2$ for a closed orientable $3$-manifold $M_1$, where $g \geq 2$, and let $(V_2, W_2; \Sigma_2)$ be the genus-$1$ Heegaard splitting for $S^2 \times S^1$.
If $(V, W ; \Sigma)$ is the splitting for $M_1 \# (S^2 \times S^1)$
obtained from $(V_1, W_1; \Sigma_1)$ and $(V_2, W_2; \Sigma_2)$, then
the sphere complex $\mathcal{H}$ for the splitting $(V, W; \Sigma)$
is isomorphic to the complex $\mathcal{A}^*_{g-1, 2}$, and hence it is a $(4g - 5)$-dimensional contractible complex. 
\end{corollary}

Recalling that the Hempel distance of the genus-$1$ Heegaard splitting of $S^3$ is $\infty$, we also have the following.
\begin{corollary}
\label{cor:contractibility of sphere compex}
Let $(V, W ; \Sigma)$ be the genus-$2$ Heegaard splitting for $S^2 \times S^1$.
Then the sphere complex $\mathcal{H}$ for the splitting $(V, W ; \Sigma)$ is a $3$-dimensional contractible complex. 
\end{corollary}

\section{Goeritz groups}
\label{sec:Goeritz groups}

Let $M$ be an orientable manifold.
Let $X_1$, $X_2, \ldots,$ $X_n$ and $Y$ be subspaces of $M$.
We denote by
\[\Homeo_+ (M, X_1, X_2, \ldots, X_n ~{\mathrm{rel}}~ Y)\]
the group of orientation-preserving
homeomorphisms of $M$ that preserve each of the subspaces $X_1$, $X_2, \ldots,$ $X_n$ setwise, 
and $Y$ pointwise.
We equip this group with the compact-open topology.
Let $\Homeo_0 (M, X_1, X_2, \ldots, X_n ~{\mathrm{rel}}~ Y)$ be the connected component of
$\Homeo_+ (M, X_1, X_2, \ldots, X_n ~{\mathrm{rel}}~ Y)$
containing the identity.
This component is a normal subgroup, and we denote by $\mathrm{MCG}_+ (M, X_1,X_2, \ldots, X_n ~{\mathrm{rel}}~ Y)$ the quotient group
\begin{eqnarray*}
\Homeo_+ (M , X_1, X_2, \ldots, X_n ~{\mathrm{rel}}~ Y) / \Homeo_0 (M, X_1, X_2, \ldots, X_n ~{\mathrm{rel}}~ Y) .
\end{eqnarray*}
Let $(V, W; \Sigma)$ be a Heegaard splitting of a closed orientable 3-manifold $M$.
We recall that the Goeritz group of the splitting $(V, W; \Sigma)$ is the group of isotopy classes of the orientation-preserving homeomorphisms of $M$ that preserve $V$ and $W$ setwise.
We denote by $\mathcal{G} (V, W; \Sigma)$ the Goeritz group, which is identified with the quotient group $\MCG_+ (M, V)$.
We note that there are natural injective homomorphisms
$\MCG_+(V) \to \MCG_+(\Sigma)$ and $\MCG_+(W) \to \MCG_+(\Sigma)$,
which can be obtained by restricting homeomorphisms of $V$ and $W$ to
$\Sigma$, respectively.
Once we regard the groups $\MCG_+(V)$ and $\MCG_+(W)$ as subgroups of $\MCG_+(\Sigma)$
with respect to the inclusions,
$\mathcal{G} (V, W; \Sigma)$ is identified with $\MCG_+(V) \cap \MCG_+(W)$.
We also note that the group $\mathcal{G} (V, W; \Sigma)$ acts on the sphere complex
$\mathcal{H}$ of $(V, W; \Sigma)$ simplicially if the splitting $(V, W; \Sigma)$ admits Haken spheres.

In \cite{Nam07}, Namazi showed that if the Hempel distance of the splitting $(V, W; \Sigma)$ is sufficiently high,
then $\mathcal{G} (V, W; \Sigma)$ is a finite group.
Later, Johnson \cite{Joh10} improved this result as follows.
\begin{theorem}[Johnson \cite{Joh10}]
\label{thm:Hempel distance and the Goeritz groups}
If the Hempel distance of the splitting $(V, W; \Sigma)$ is at least $4$,
then the group $\mathcal{G} (V, W; \Sigma)$ is finite.
\end{theorem}
For Heegaard splittings of low Hempel distance, the situation is much more complicated
as mentioned in Introduction.

In this section, we are interested in the Goeritz groups of the Heegaard splittings described in Section \ref{sec:Sphere complexes}. 
Let $(V, W; \Sigma)$ be a genus-$g$ Heegaard splitting of a closed orientable $3$-manifold $M$, where $g \geqslant 2$.
Suppose that there exists a unique reducing disk $E_0$ in $V$.
Fix a Haken sphere $P$ for the splitting $(V, W; \Sigma)$ which represents a vertex of the complex $\mathcal H_{E_0}$. 
That is, $P$ is the Haken sphere determined by a simple closed curve in $\Sigma$ intersecting $\partial E_0$ in a single point as in Section \ref{sec:Sphere complexes}.
Then the disk $P \cap V$ cut off from $V$ a
solid torus whose meridian disk is $E_0$.

The handlebody $V$ cut off by $P \cap V$
consists of two handlebodies $V_1'$ and $V_2'$, and similarly $W$ cut off by $P \cap W$ consists of $W_1'$ and $W_2'$.
Gluing $3$-balls $B_1$ and $B_2$ on $V_1' \cup W_1'$ and $V_2' \cup W_2'$ along $P$, we obtain two Heegaard splittings $(V_1, W_1; \Sigma_1)$ and $(V_2, W_2; \Sigma_2)$ respectively.
We may assume that $(V_2, W_2; \Sigma_2)$ is the genus-$1$ splitting of $S^2 \times S^1$, while $(V_1, W_1; \Sigma_1)$ is the genus-$(g-1)$ splitting of a $3$-manifold having no $S^2 \times S^1$ summand in its prime decomposition.


Suppose that the Goeritz group $\mathcal{G}(V_1 , W_1; \Sigma_1)$ is generated by
finitely many elements $\omega_1$, $\omega_2 , \ldots, \omega_m$.
For each $i \in \{1, 2, \ldots, m \}$, the element $\omega_i$ has a representative homeomorphism $w_i \in \Homeo_+ (M_1, V_1)$
satisfying $w_i|_{B_i}$ is the identity.
Thus there exists an element $\tilde{\omega}_i$ of $\mathcal{G}(V , W; \Sigma)$
represented by a homeomorphism $\tilde{w}_i \in \Homeo_+ (M, V)$ such that
$\tilde{w}_i (P) = P$,
$\tilde{w}_i |_{V_1' \cup W_1'} = w_i|_{V_1' \cup W_1'}$, and
$\tilde{w}_i |_{V_2' \cup W_2'}$ is the identity.

We also define the elements $\lambda_j$ and $\mu_j$ for each $j \in \{ 1 , 2 , \ldots , g-1 \}$, and the elements $\beta$ and $\epsilon$ of
$\mathcal{G}(V, W ; \Sigma)$ as follows.
The elements $\lambda_j$ and $\mu_j$ have representative homeomorphisms obtained by
pushing $V_2' \cup W_2'$ so that $P \cap \Sigma$ moves along the arcs depicted in
Figure \ref{fig:pushing1} respectively.

\begin{figure}[htbp]
\includegraphics[width=12cm,clip]{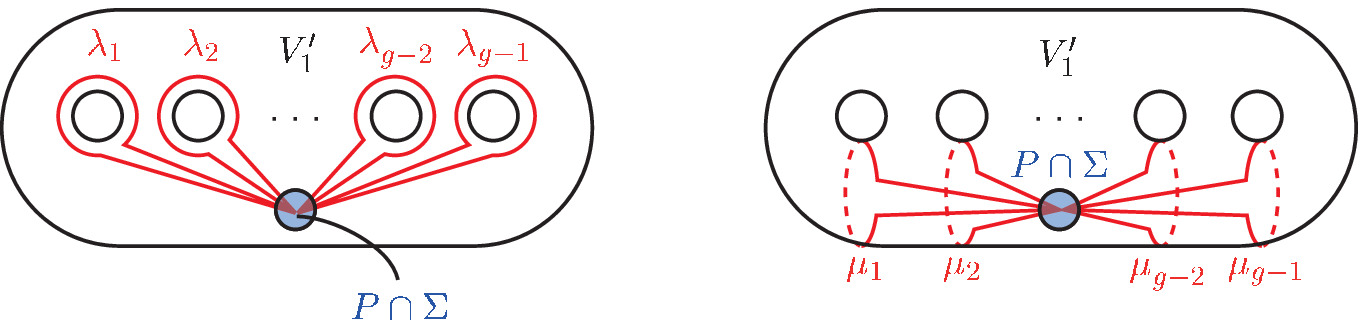}
\caption{}
\label{fig:pushing1}
\end{figure}

The element $\beta$ is defined by extending a half-Dehn twist about the disk $P \cap V$, and the element $\epsilon$ is defined by extending a Dehn twist about the unique reducing disk $E_0$ in $V$.
See Figure \ref{fig:beta}.
Note that all of $\tilde{\omega}_i$, $\lambda_j$, $\mu_j$, $\beta$ and $\epsilon$ preserve the equivalence class of the Haken sphere $P$.

\begin{figure}[htbp]
\begin{center}
\includegraphics[width=7cm,clip]{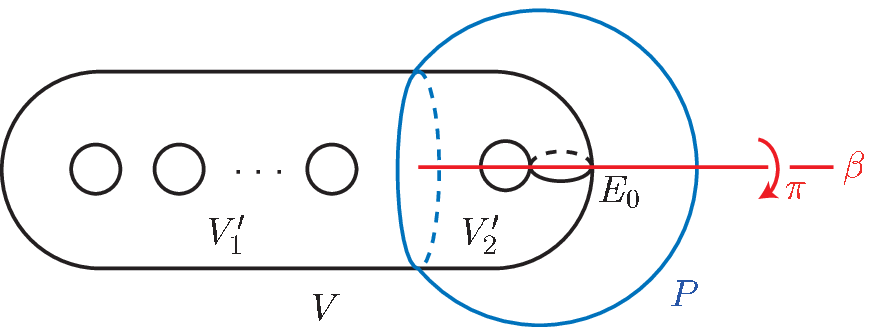}
\caption{}
\label{fig:beta}
\end{center}
\end{figure}

\begin{lemma}
\label{lem:stabilizer of a Haken sphere}
Under the setting in the above,
the subgroup of $\mathcal{G} (V, W; \Sigma)$ consisting of elements
that preserve the equivalence class of $P$ is generated by
$\tilde{\omega}_i$, $\mu_j$, $\lambda_j$,
$\beta$ and $\epsilon$, where $i \in \{ 1 , 2 , \ldots, m \}$ and $j \in \{ 1 , 2 , \ldots,  g - 1  \}$.
\end{lemma}
\begin{proof}
Let $m_j$, $l_j$, $b$ and $e$ be representative homeomorphisms of
$\mu_j$, $\lambda_j$, $\beta$ and $\epsilon$, respectively, preserving $P$.
We may assume that each of $m_j$, $l_j$ and $b^2$ fixes
$V'_2 \cup W'_2$.
Let $\varphi$ be any element of $\mathcal{G} (V, W; \Sigma)$ that preserves the equivalence class of $P$, and let
$f \in \Homeo_+ (M, V)$ be one of its representatives satisfying $f (P) = P$.
We will show that $f$ is isotopic to a composition of a finite number of ${\tilde{w}_i}^{\pm 1}$
${m_j}^{\pm 1}$, ${l_j}^{\pm 1}$, $b^{\pm 1}$ and $e^{\pm 1}$ up to an isotopy preserving $V$.

Let $E'_0$ be an essential disk in $W$ bounded by the unique reducing disk $\partial E_0$ in $V$.
Composing $f$ with a power of $b$, if necessary, and by an appropriate isotopy preserving $V$, we get a
map $f_1 \in \Homeo_+ (M, V)$ fixing $E_0 \cup E_0'$ and $P$.
Moreover, by composing $f_1$ with a power of $e$, if necessary, and by an appropriate isotopy preserving $V$, we get a
map $f_2 \in \Homeo_+ (M, V)$ fixing $E_0 \cup E_0'$, $\partial V'_2$ and $\partial W'_2$.
Note that the union of $E_0 \cup E_0'$ and $\partial V'_2 \cap \partial W'_2$ cuts $V'_2 \cup W'_2$ into two 3-balls.
Thus, by Alexander's trick, we may assume that $f_2$ fixes $V_2' \cup W_2'$.

Suppose first that $g \geqslant 3$.
Let $D_1$ be the disk $\Sigma_1 \cap B_1$ and
choose a point $p_1$ in the interior of $D_1$.
By the Birman exact sequence \cite{Bir74}, we have the following commutative diagrams:
\[\xymatrix{
1 \ar[r] &
\pi_1 (\Sigma_1 , p_1) \ar[r]\ar[d]^{=}\ar@{}[dr]|\circlearrowleft
& \MCG_+(M_1, V_1, p_1) \ar[r]\ar[d]\ar@{}[dr]|\circlearrowleft &
\MCG_+(M_1, V_1) \ar[r]\ar[d] & 1\\
1 \ar[r] & \ar[r]^-{\mathrm{push}}  \pi_1 (\Sigma_1 , p_1)
\ar[r] & \ar[r]^-{\mathrm{forget}} \MCG_+(\Sigma_1 , p_1)  &
 \MCG_+( \Sigma_1 ) \ar[r] & 1 ,
} \]
and
\[\xymatrix{
1 \ar[r] &
\Integer \ar[r]\ar[d]^{=}\ar@{}[dr]|\circlearrowleft &  \MCG_+(M_1, V_1 ~\mathrm{rel}~ D_1) \ar[r]\ar[d]\ar@{}[dr]|\circlearrowleft &
\MCG_+(M_1, V_1 , p_1) \ar[r]\ar[d] & 1\\
1 \ar[r] & \ar[r] \Integer
\ar[r] & \MCG_+( \Sigma_1 ~\mathrm{rel}~D_1) \ar[r] &
 \MCG_+( \Sigma_1 , p_1) \ar[r] & 1 	.
} \]
In these diagrams, each vertical arrow is an injective homeomorphism.
In the first diagram, the arrow ``$\xrightarrow{\mathrm{push}}$" implies the {\it pushing map} and
``$\xrightarrow{\mathrm{forget}}$" implies the {\it forgetful map}.
The group $\Integer$ in the second diagram is generated by the Dehn twist about the disk $D_1$.
See for instance \cite{FM12, HH12}.
By the assumption, the group $\MCG_+(M_1 , V_1) = \mathcal{G}(V_1, W_1 ; \Sigma_1)$ is generated by
$\omega_1, \omega_2, \cdots, \omega_m$.
The image of $\pi_1 (\partial V_1, p_1)$ in $\MCG_+((M_1, V_1, p_1)$ is the subgroup
generated by the elements whose representatives correspond to
$m_j|_{V'_1 \cup W'_1}$ and $l_j|_{V'_1 \cup W'_1}$, where $j \in \{ 1, 2, \ldots, g-1 \}$.
Moreover a generator of $\Integer$ in the second diagram
corresponds to $b^2|_{V'_1 \cup W'_1}$.
Therefore, by the above diagrams and
a natural identification
\[\MCG_+(M_1 , V_1 ~\mathrm{rel}~ D_1) \cong \MCG_+(M , V ~\mathrm{rel}~ V'_2 \cup W'_2), \]
it follows that
$f_2$ can be written as a composition of
a finite number of
$\tilde{w}_i$ $(i \in \{ 1 , 2 , \ldots, n \})$, ${m_j}^{\pm 1}$, ${l_j}^{\pm 1}$
$(j \in \{ 1 , 2 , \ldots, g - 1 \})$ and $b^{\pm 2}$
up to isotopy preserving $V$.

Suppose that $g = 2$.
Then instead of the first diagram in the above argument,
we have the following simpler diagram:
\[\xymatrix{
1 \ar[r] & \ar[r]^-{\cong}
\MCG_+(M_1, V_1, p_1) \ar[d]\ar@{}[dr]|\circlearrowleft &
\MCG_+(M_1, V_1) \ar[r]\ar[d] & 1\\
1 \ar[r] & \ar[r]^-{\cong} \MCG_+( \Sigma_1 , p_1) &
 \MCG_+( \Sigma_1 ) \ar[r] & 1 .
} \]
Hence $f_2$ can be written as the composition of
a finite number of
$\tilde{w}_i$ $(i \in \{ 1 , 2 , \ldots, n \})$ and $b^{\pm 2}$
up to isotopy preserving $V$.
This completes the proof.
\end{proof}

In addition to the elements $\tilde{\omega}_i$, $\mu_j$, $\lambda_j$,
$\beta$ and $\epsilon$, we define the elements $\lambda_j^*$ and $\mu_j^*$ of $\mathcal{G} (V, W; \Sigma)$ for each $j \in \{ 1 , 2 , \ldots , g-1 \}$ as follows.
Let $V^*$ be the handlebody $V$ cut off by the unique reducing disk $E_0$.
Let $E_0^+$ and $E_0^-$ be disks in $\partial V^*$ coming from $E_0$.
The elements $\lambda_j^*$ and $\mu_j^*$ for each $j  \in \{ 1 , 2 , \ldots , g-1 \}$ are defined by
pushing $E_0^+$ along the arcs depicted in
Figure \ref{fig:pushing2}.
Each of these maps is realized by sliding a foot of the 1-handles
$\Nbd (E_0; V) $ and $\Nbd (E_0'; W) $
of $V$ and $W$, respectively, where $E_0'$ is a disk  in $W$ bounded by $\partial E_0$.
We observe that, for any simple arcs $\gamma$ and $\gamma'$ on $\partial V^*$ connecting $\partial E_0^+$ and $\partial E_0^-$, 
there exists an element $\varphi$ of $\mathcal{G} (V, W; \Sigma)$, which is a finite product of $\beta$,  $\lambda_j^*$ and $\mu_j^*$ for $j \in \{ 1 , 2 , \ldots , g-1 \}$, such that $\varphi$ has a representative map sending $\gamma'$ to $\gamma$.

\begin{figure}[htbp]
\begin{center}
\includegraphics[width=12cm,clip]{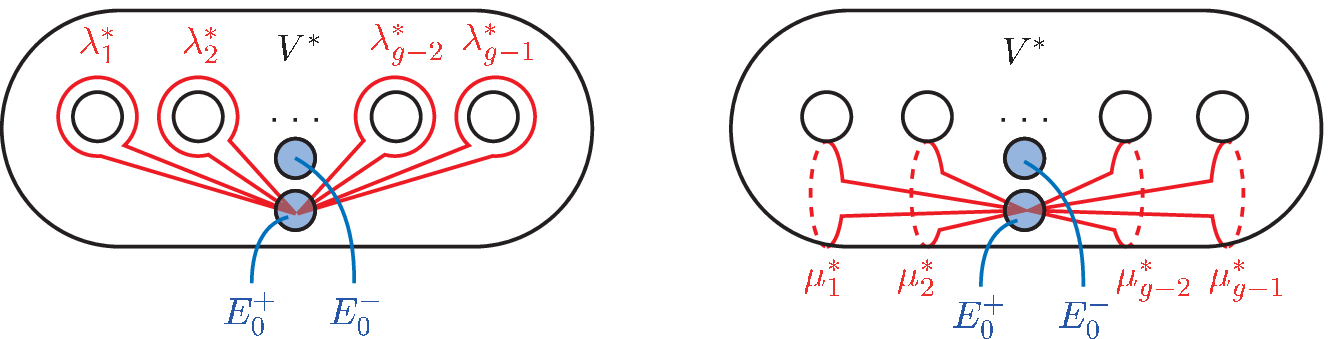}
\caption{}
\label{fig:pushing2}
\end{center}
\end{figure}

Now let $\psi$ be any element of $\mathcal{G} (V, W; \Sigma)$ represented by a map $k \in \Homeo_+ (M, V)$.
Then $k(P)$ is also a Haken sphere representing a vertex of the complex $\mathcal H_{E_0}$.
If $k(P)$ is equivalent to $P$, then $\psi$ preserves the equivalence class of $P$ and is a finite product of the elements $\tilde{\omega}_i$, $\mu_j$, $\lambda_j$,
$\beta$ and $\epsilon$ by Lemma \ref{lem:stabilizer of a Haken sphere}.
Suppose that $k(P)$ is not equivalent to $P$.
We may consider $P$ and $k(P)$ as the Haken spheres determined by simple arcs $\gamma$ and $\gamma'$ on $\partial V^*$ respectively, each of which connects $\partial E_0^+$ and $\partial E_0^-$.
Then there exists a finite product, say $\varphi$, of $\beta$,  $\lambda_j^*$ and $\mu_j^*$ 
such that $\varphi$ has a representative map sending $\gamma'$ to $\gamma$.
This map also sends $k(P)$ to $P$ up to isotopy.
Thus the composition $\varphi \psi$ preserves the equivalence class of $P$, and consequently $\varphi$ is a finite product of
$\tilde{\omega}_i$, $\mu_j$, $\lambda_j$, $\lambda^*_j$, $\mu^*_j$, $\beta$ and $\epsilon$.
We summarize this observation as follows.

\begin{theorem}
\label{thm:generators of Goeritz groups}
Let $(V, W; \Sigma)$ be the Heegaard splitting obtained from a genus-$(g-1)$ splitting $(V_1, W_1; \Sigma_1)$ for a $3$-manifold and the genus-$1$ splitting $(V_2, W_2; \Sigma_2)$ for $S^2 \times S^1$, where $g \geqslant 2$.
Suppose that there exists a unique reducing disk $E_0$ in $V$.
If the Goeritz group of $(V_1 , W_1; \Sigma_1)$ is finitely generated, then the Goeritz group of $(V, W; \Sigma)$
is also finitely generated.
Moreover, under the setting described above, the Goeritz group of $(V, W; \Sigma)$
is generated by
$\tilde{\omega}_i$, $\mu_j$, $\lambda_j$, $\lambda^*_j$, $\mu^*_j$, $\beta$ and $\epsilon$,
where $i \in \{ 1 , 2 , \ldots, m \}$ and $j \in \{ 1 , 2 , \ldots, g-1 \}$.
\end{theorem}

By Proposition \ref{prop:uniqueness of the reducing disk}
and
Theorems \ref{thm:Hempel distance and the Goeritz groups} and
\ref{thm:generators of Goeritz groups} we have the following:

\begin{corollary}
\label{cor:generators of Goeritz groups for the connected sums}
Let $(V_1, W_1; \Sigma_1)$ be a genus-$(g-1)$ Heegaard splitting of Hempel distance at least $4$
for a closed orientable $3$-manifold $M_1$, where $g \geqslant 2$, and let $(V_2, W_2; \Sigma_2)$ be the genus-$1$ Heegaard splitting for $S^2 \times S^1$.
If $(V, W ; \Sigma)$ is the splitting for $M_1 \# (S^2 \times S^1)$
obtained from $(V_1, W_1; \Sigma_1)$ and $(V_2, W_2; \Sigma_2)$, then the Goeritz group of the splitting $(V, W; \Sigma)$ is finitely generated.
\end{corollary}

We note that Corollary \ref{cor:generators of Goeritz groups for the connected sums}  
implies, in particular, that 
the Goeritz group of the genus-$2$ Heegaard splitting for $S^2 \times S^1$ is finitely generated, 
which is shown in \cite{CK13a}.

\section*{Acknowledgments}
This work was carried out while the second-named author was visiting
Universit\`a di Pisa as a
JSPS Postdoctoral Fellow for Research Abroad.
He is grateful to the university and its staffs for
the warm hospitality.


\begin{thebibliography}{99999}
\bibitem{Akb08}
Akbas, E.,
A presentation for the automorphisms of the 3-sphere that preserve a genus
two Heegaard splitting, Pacific J. Math. {\bf 236} (2008), 201--222.

\bibitem{Bir74}
Birman, J. S.,
\emph{Braids, links and mapping class groups},
Ann. Math. Stud. {\bf 82}, Princeton Univ. Press, Princeton, N.J., 1974.



\bibitem{Cho08}
Cho, S.,
Homeomorphisms of the 3-sphere that preserve a Heegaard splitting of genus
two, Proc. Amer. Math. Soc. {\bf 136} (2008), 1113--1123.

\bibitem{Cho12}
Cho, S.,
Genus two Goeritz groups of lens spaces, Pacific J. Math. {\bf 265} (2013), 1--16.

\bibitem{CK12}
Cho, S., Koda, Y.,
Primitive disk complexes for lens spaces, arXiv:1206.6243.


\bibitem{CK13a}
Cho, S., Koda, Y.,
The genus two Goeritz group of $S^2 \times S^1$, 
Math. Res. Lett. {\bf 21} (2014) no. 3, 449--460.

\bibitem{CK13b}
Cho, S., Koda, Y.,
Disk complexes and genus two Heegaard splittings for non-prime 3-manifolds, 
Int. Math. Res. Not. IMRN, First published online: 2014,  
doi: 10.1093/imrn/rnu061. 

\bibitem{CMS09}
Cho, S., McCullough, D., Seo, A.,
Arc distance equals level number,
Proc. Amer. Math. Soc. {\bf 137} (2009), 2801--2807.

\bibitem{FM12}
Farb, B., Margalit, D.,
{\it A primer on mapping class groups},
Princeton Mathematical Series {\bf 49},
Princeton University Press, Princeton, NJ, 2012. 

\bibitem{Goe33}
Goeritz, L.,
Die Abbildungen der Brezelfl\"{a}che und der Vollbrezel vom Geschlecht $2$, 
Abh. Math. Semin. Hamb. Univ. {\bf 9} (1933), 244--259.





\bibitem{HH12}
Hamenst\"{a}dt, U., Hensel, S.,
The geometry of the handlebody groups I: distortion,
J. Topol. Anal.  {\bf 4}  (2012),  71--97. 

\bibitem{Hak68}
Haken, W., 
Some results on surfaces in 3-manifolds, Studies in Modern Topology, 
Math. Assoc. Amer., Prentice-Hall, Englewood Cliffs, N.J., 1968, 39--98.    


\bibitem{Hat91}
Hatcher, A.,
On triangulations of surfaces,
Topology Appl.  {\bf 40}  (1991),  189--194.

\bibitem{Hem01}
Hempel, J.,
3-manifolds as viewed from the curve complex,
Topology {\bf 40} (2001), 631--657.

\bibitem{IM10}
Irmak, E., McCarthy J. D.,
Injective simplicial maps of the arc complex, Turkish J. Math.  {\bf 34}  (2010),  339--354.

\bibitem{Joh10}
Johnson, J.,
Mapping class groups of medium distance Heegaard splittings,
Proc. Amer. Math. Soc. {\bf 138} (2010), 4529--4535.

\bibitem{Joh11}
Johnson, J.,
Mapping class groups of once-stabilized Heegaard splittings,  arXiv:1108.5302.

\bibitem{KP10}
Korkmaz, M., Papadopoulos, A.,
On the arc and curve complex of a surface,
Math. Proc. Cambridge Philos. Soc.  {\bf 148}  (2010),  473--483.

\bibitem{Lei05}
Lei, F.,
Haken spheres in the connected sum of two lens spaces,
Math. Proc. Camb. Philos. Soc. {\bf 138} (2005), 97--105.

\bibitem{LZ04}
Lei, F., Zhang, Y.,
Haken spheres in genus 2 Heegaard splittings
of nonprime 3-manifolds,
Topology Appl. {\bf 142} (2004) 101--111.

\bibitem{McC91}
McCullough, D.,
Virtually geometrically finite mapping class groups of $3$-manifolds,
J. Diff. Geom. {\bf 33} (1991), 1--65.


\bibitem{Nam07}
Namazi, H.,
Big Heegaard distance implies finite mapping class group, Topology Appl. {\bf 154} (2007), 2939--2949.



\bibitem{ST03}
Scharlemann, M., Thompson, A.,
Unknotting tunnels and Seifert surfaces, Proc. London Math. Soc. (3)
{\bf 87} (2003),
523--544.

\bibitem{Sch04}
Scharlemann, M.,
Automorphisms of the $3$-sphere that preserve a
genus two Heegaard splitting, Bol. Soc. Mat. Mexicana {\bf 10} (2004),
503--514.

\bibitem{Seo08}
Seo, A.,
Torus leveling of $(1,1)$-knots,
dissertation at the University of Oklahoma (2008).






\bibitem{Wal68}
Waldhausen, F., Heegaard-Zerlegungen der 3-Sph\"are, Topology {\bf 7}
(1968), 195--203.



\end{thebibliography}
\end{document}